\documentclass[11pt]{amsart}
\usepackage{amsthm,amsmath,amssymb,enumerate,mathtools}
\newtheorem{theorem}{Theorem}[section]
\newtheorem{claim}{}[theorem]
\newtheorem{lemma}[theorem]{Lemma}
\newtheorem{problem}[theorem]{Problem}
\newtheorem{proposition}[theorem]{Proposition}

\theoremstyle{definition}

\newcommand{\bR}{\mathbb R}
\newcommand{\bC}{\mathbb C}
\newcommand{\bZ}{\mathbb Z}

\newcommand{\cC}{\mathcal{C}}

\newcommand{\cF}{\mathcal{F}}

\newcommand{\cL}{\mathcal{L}}

\newcommand{\cM}{\mathcal{M}}

\newcommand{\cU}{\mathcal{U}}

\DeclareMathOperator{\si}{si}

\DeclareMathOperator{\cl}{cl}

\DeclareMathOperator{\GF}{GF}
\DeclareMathOperator{\AG}{AG}

\DeclareMathOperator{\rank}{rank}
\newcommand{\elem}{\varepsilon}
\newcommand{\del}{\!\setminus\!}

\usepackage{environ, comment}
\usepackage[labelsep=period]{caption}

\newcommand{\U}{U_{2,4} \oplus U_{2,4}}

\title{$2$-Modular Matrices}

\author{James Oxley}
\address{Mathematics Department, Louisiana State University, Baton Rouge, Louisiana, USA}
\email{oxley@math.lsu.edu}

\author{Zach Walsh}
\address{Mathematics Department, Louisiana State University, Baton Rouge, Louisiana, USA}
\email{walsh@lsu.edu}



\date{\today}

\begin{document}

\maketitle

\begin{abstract}
	An integer matrix $A$ is $\Delta$-modular if the determinant of each $\rank(A) \times \rank(A)$ submatrix has absolute value at most $\Delta$. 
	The class of $1$-modular, or unimodular, matrices is of fundamental significance in both integer programming theory and matroid theory.
	A 1957 result of Heller shows that the maximum number of nonzero, pairwise non-parallel columns of a rank-$r$ unimodular matrix is ${r + 1 \choose 2}$.
	We prove that, for each sufficiently large integer $r$, the maximum number of nonzero, pairwise non-parallel columns of a rank-$r$ $2$-modular matrix is ${r + 2 \choose 2} - 2$.
	
\end{abstract}

% introduction

\begin{section}{Introduction}
This paper is motivated by the theory of integer programming.
Given a \emph{constraint matrix} $A\in \bZ^{m\times n}$ and vectors $b\in \bZ^m$ and $c\in \bZ^n$, a fundamental problem of integer programming is to find an integer vector $x$ in the polyhedron $\{x\in \bR^n\colon Ax\le b\}$ that maximizes the inner product $c^Tx$.
In general, one cannot expect to optimize an integer program efficiently, but better efficiency is possible if the constraint matrix has special structure.

Given a positive integer $\Delta$, an integer matrix $A$ is \emph{$\Delta$-modular} if the determinant of each $\rank(A)\times \rank(A)$ submatrix of $A$ has absolute value at most $\Delta$. 
A $1$-modular matrix is \emph{unimodular}, and integer programs with a unimodular matrix can be optimized in strongly polynomial time, because every vertex of the polytope $\{x\in \bR^n\colon Ax\le b\}$ is integral \cite{HK}.
This fact has led to considerable interest in the complexity of integer programming with a $\Delta$-modular constraint matrix (see, for example, \cite{VC} and \cite{Artmann+}).
This complexity is unknown when $\Delta \ge 3$, but a recent breakthrough of Artmann, Weismantel, and Zenklusen \cite{AWZ} shows that integer programs with a $2$-modular, or \emph{bimodular}, constraint matrix can be optimized in strongly polynomial time. 
In view of this result, the structural properties of $2$-modular matrices are of significant interest.
A well-known result of Heller~\cite{Heller} shows that a rank-$r$ unimodular matrix has at most ${r + 1 \choose 2}$ nonzero, pairwise non-parallel columns.
 We prove the analogous result for $\Delta = 2$.

 \begin{theorem} \label{main}
 For each sufficiently large integer $r$, the maximum number of nonzero, pairwise non-parallel columns of a rank-$r$ $2$-modular matrix is ${r + 2 \choose 2} - 2$.
 \end{theorem}

This was proven independently by Lee et al. \cite{LPSX}, who showed that this bound holds for all $r$.
The previous best upper bound was ${r + 1 \choose 2} + C\cdot r$ for a large constant $C$, due to Geelen, Nelson, and Walsh~\cite{complex}.
The bound of Theorem~\ref{main} is tight, for example, for the matrices $A_r$ and $A'_r$ in Figure \ref{matrices}, where $D_r$ is the $r \times {r \choose 2}$ matrix whose columns consist of all $r$-tuples with exactly two nonzero entries, the first equal to $1$ and the second equal to $-1$.

\begin{figure}
\begin{center}
		$A_r = \left[ \begin{tabular}{ccc|ccc|ccc}
			   &  &  &  &  &   & $1$ & $\cdots$ & $1$  \\
			   \cline{7-9}
			   &  &  &  &  &  &  &  &   \\
			   &  $I_{r}$  &  &  &  $D_{r}$  &  &  & $I_{r-1}$ & \\
			   &  &  &  &  &  &  &  &  \\
			   &  &  &  &  &  &  &  &  \\
		\end{tabular} \right]$
	\end{center} 
	
\begin{center}
		$A'_r = \left[ \begin{tabular}{ccc|ccc|ccc|c}
			   &  &  &  &  &   & $1$ & $\cdots$ & $1$ & $1$ \\
			     &  &  &  &  &    & $1$ & $\cdots$ & $1$ & $1$ \\
			\cline{7-10}
			   &  $I_{r}$ &  &  & $D_{r}$  &  &  & & & $0$\\
			   &  &  &  &  &   &  &  $-I_{r-2}$ & & $\vdots$ \\
			   &  &  &  &  &  &  &  &  & $0$ \\
		\end{tabular} \right]$
	\end{center} 
	\caption{$A_r$ and $A'_r$.}
	\label{matrices}
	\end{figure}

We prove Theorem~\ref{main} using matroids.
For each positive integer $\Delta$, let $\cM_{\Delta}$ denote the class of matroids with a representation over $\bR$ as a $\Delta$-modular matrix.
We say that such a matroid is \emph{$\Delta$-modular}.
An integer matrix $A$ is \emph{totally $\Delta$-modular} if the determinant of every submatrix of $A$ has absolute value at most $\Delta$.
We define $\overline \cM_{\Delta}$ to be the class of matroids with a representation over $\bR$ as a totally $\Delta$-modular matrix.
Clearly $\overline \cM_{\Delta} \subseteq \cM_{\Delta}$ for each $\Delta \ge 1$, because every totally $\Delta$-modular matrix is also $\Delta$-modular.
Since every unimodular matrix is row-equivalent to a totally unimodular matrix, the classes $\overline \cM_1$ and $\cM_1$ are in fact equal.

The class $\cM_1$ is the class of \emph{regular matroids}.
Tutte~\cite{Tutte} proved that this class coincides with the class of matroids representable over every field, and that there are exactly three minor-minimal matroids not in $\cM_1$.
For $\Delta \ge 2$, the class $\cM_{\Delta}$ is not very well-studied.
It is proved in \cite{complex} that $\cM_{\Delta}$ is minor-closed, and that every matroid in $\cM_{\Delta}$ is representable over every field with characteristic greater than $\Delta$.
In particular, every matroid in $\cM_2$ is representable over $\GF(3)$ and $\GF(5)$; such a matroid is called \emph{dyadic}.

In view of Tutte's results and the fundamental significance of unimodular matrices, there are several natural open questions concerning $\cM_{\Delta}$.

\begin{problem} \label{open}
Let $\Delta$ be a positive integer.
\begin{enumerate}[$(i)$]
\item What is the maximum size of a simple rank-$r$ matroid in $\cM_{\Delta}$, and which matroids attain this maximum?

\item What are the minor-minimal matroids that are not in $\cM_{\Delta}$?

\item Is $\overline \cM_{\Delta}$ equal to $\cM_{\Delta}$?

\item Is $\overline \cM_{\Delta}$ closed under duality?
\end{enumerate}
\end{problem}

Theorem~\ref{main} shows that the maximum size of a simple rank-$r$ $2$-modular matroid is ${r + 2 \choose 2} - 2$, if $r$ is sufficiently large, and the vector matroids of $A_r$ and $A'_r$ show that there are at least two matroids that attain this maximum.
In order to prove Theorem~\ref{main}, we show that $(iii)$ and $(iv)$ have affirmative answers for $\Delta = 2$, and we find several excluded minors for $\cM_2$, including $U_{2,5}$, $F_7$, $\U$, and $\AG(2,3)\del e$.
In fact, we prove the following result, which is stronger than Theorem~\ref{main}.
A \emph{line} in a matroid is a rank-$2$ flat, and a line is \emph{long} if it contains at least three \emph{points}, that is, rank-one flats.
The matroid $R_9$ is the ternary Reid geometry \cite[p.~654]{Oxley}, which is the simple rank-$3$ matroid consisting of long lines $L_1, L_2, L_3$ with a common intersection point $x$ so that $|L_1| = |L_2| = 4$,  $|L_3| = 3$, and both elements in $L_3 - \{x\}$ are on four long lines.

\begin{theorem} \label{main matroids}
For each sufficiently large integer $r$, the maximum size of a simple rank-$r$ matroid with no minor in $\{U_{2,5}, F_7, \U, R_9\}$ is ${r + 2 \choose 2} - 2$.
\end{theorem}

While we make progress on Problem \ref{open} for $\Delta = 2$, it is surprisingly difficult for $\Delta \ge 3$.
For fixed $\Delta \ge 3$, the current best upper bound on the maximum size of a simple rank-$r$ matroid in $\cM_{\Delta}$ is ${r + 1 \choose 2} + C \cdot r$ for a large constant $C$, proved by Geelen, Nelson, and Walsh~\cite{complex}.
Recently, Lee et al. \cite{LPSX} proved an upper bound of $\Delta^2{r + 1 \choose 2}$, which is the first bound that is polynomial in both $r$ and $\Delta$.
Both of these are recent additions to a long sequence of previous results; we direct the reader to \cite{Glanzer} for more details.
However, it is unclear what the correct bound should be for $\Delta \ge 3$.
It is also unclear in general what the rank-$2$ uniform excluded minor is for $\cM_{\Delta}$.
On a more positive note, it follows from a result of D'Adderio and Moci \cite[Theorem 2.2]{DM} that $\cM_{\Delta}$ is closed under duality for all positive $\Delta$.
In Theorem \ref{duality} we present an elementary proof of this result due to Marcel Celaya (private communication).
However, for $\Delta \ge 3$ it is unknown whether $\overline \cM_{\Delta} = \cM_{\Delta}$, and whether $\overline \cM_{\Delta}$ is closed under duality.
Some of the difficulties of these open problems are discussed in Section \ref{duality}.
After introducing some preliminaries in Section \ref{preliminaries}, we show in Section \ref{max matroids} that $A_r$ and $A'_r$ attain equality in Theorem \ref{main}.
The proofs of Theorems \ref{main} and \ref{main matroids} are given in Section \ref{proofs}.
A crucial case in the proof of the latter, when $M$ has a spanning-clique restriction, is treated in Section \ref{spanning-clique case}.
Prior to that, Sections \ref{1-extensions}, \ref{2-extensions}, and \ref{k-extensions} consider $1$-, $2$-, and $k$-element extensions of a clique.
\end{section}

% preliminaries

\begin{section}{Preliminaries} \label{preliminaries}
We follow the notation of \cite{Oxley}, unless otherwise stated.
For a matroid $M$, we write $|M|$ and $\elem(M)$ for, respectively, $|E(M)|$ and the number of points of $M$.
For an integer $\ell \ge 2$, we denote by $\cU(\ell)$ the class of matroids with no $U_{2,\ell+2}$-minor.

We first describe a construction that shows how the vector matroids of $A_r$ and $A'_r$ are related to $M(K_{r+2})$.
If a matroid $M$ is obtained from a matroid $N$ by deleting a non-empty subset $T$ of $E(N)$, then $N$ is an \emph{extension} of $M$.
If $|T| = 1$, then $N$ is a \emph{single-element extension} of $M$, and if $|T| = k$ for $k \ge 2$, then $N$ is a \emph{$k$-element extension} of $M$.
The single-element extensions of a matroid $M$ are in one-to-one correspondence with the \emph{modular cuts} of $M$, which are the sets $\cF$ of flats with the following two properties:
\begin{enumerate}[(i)]
\item If $F \in \cF$ and $F'$ is a flat of $M$ containing $F$, then $F' \in \cF$.

\item If $F_1, F_2 \in \cF$ and $r_M(F_1) + r_M(F_2) = r_M(F_1 \cup F_2) + r_M(F_1 \cap F_2)$, then $F_1 \cap F_2 \in \cF$.
\end{enumerate}
A pair of flats that satisfies (ii) is a \emph{modular pair}.
A modular cut $\cF$ is \emph{proper} if it does not contain all flats of $M$.
The modular cut of $M$ corresponding to a single-element extension $N$ of $M$ by an element $e$ is the set of flats of $M$ that span $e$ in $N$.
Conversely, Crapo~\cite{Crapo} showed that, for every modular cut $\cF$ of $M$, there is a unique single-element extension $N$ of $M$ by an element $e$ for which $\cF$ is precisely the set of flats of $M$ that span $e$ in $N$ (see \cite[Theorem 7.2.3]{Oxley}).
Given a set $\cF$ of flats of $M$, the modular cut of $M$ \emph{generated} by $\cF$ is the intersection of all modular cuts of $M$ that contain every flat in $\cF$.
If $N$ is a single-element extension of $M$ by an element $e$ corresponding to the modular cut generated by a single flat $F$, then we say that $e$ has been \emph{freely added} to $F$.

A matroid $M'$ is an \emph{elementary projection} or an \emph{elementary quotient} of $M$ if there is a single-element extension $N$ of $M$ by an element $e$ so that $M' = N/e$.
Thus, every elementary projection of $M$ corresponds to a modular cut of $M$.
We prove in the following section that the vector matroids of the matrices $A_r$ and $A'_r$ are both simplifications of elementary projections of $M(K_{r+2})$. 

Two classes of matroids arise in the proof of Theorem~\ref{main matroids}: spikes and Reid geometries.
A \emph{spike} is a simple matroid $S$ with an element $t$ so that the simplification of $S/t$ is a circuit, and each parallel class of $S/t$ has size two.
The element $t$ is a \emph{tip} of $S$.
A spike $S$ may have several elements that can serve as a tip.
For example, the Fano plane is a spike, and every element can serve as a tip.
However, a spike of rank at least four has a unique tip.
We will also make use of the fact that spikes are not graphic (see~\cite[p. 662]{Oxley}).

A \emph{Reid geometry} is a simple rank-$3$ matroid $N$ consisting of long lines $L_1, L_2, L_3$ with a common intersection point $x$ so that $|L_1| = |L_2|$ and $|L_3| = 3$, while both elements in $L_3 - \{x\}$ are on $|L_1|$ long lines of $N$.
If $|L_1| = 3$, then $N$ is the Fano plane $F_7$.
If $|L_1| = 4$, then $N$ is the ternary Reid geometry $R_9$.
\end{section}

% max-sized matroids

\begin{section}{The Maximum-Sized Matroids} \label{max matroids}
In this section, we describe two matroids that meet the bound of Theorem~\ref{main matroids}.
We first show that the matrices $A_r$ and $A'_r$ are in fact $2$-modular.

\begin{lemma} \label{A_r}
For each integer $r \ge 2$, the matrix $A_r$ is $2$-modular.
\end{lemma}
\begin{proof}
Let $r$ be minimal so that the lemma is false.
Then $r > 2$.
Let $A$ be an $r \times r$ submatrix of $A_r$ so that $|\det(A)| > 2$.
If $A$ has a row with exactly one nonzero entry, then the determinant of $A$ is $\pm 1$ times the determinant of an $(r-1) \times (r-1)$ submatrix of $A_{r-1}$, and so $|\det(A)| \le 2$, by the minimality of $r$.
Then each row of $A$ has at least two nonzero entries, and, since each column of $A_r$ has at most two nonzero entries, this implies that $A$ has precisely two nonzero entries in each row and column.
By permuting rows and columns of $A$ and multiplying some rows and columns of $A$ by $-1$, we may obtain a matrix 
$A'$ for which all but at most one column has nonzero entries $1$ and $-1$.
If $A'$ has no column with two entries equal to $1$, then $|\det(A')| \le 1$.
Otherwise, by expanding along a column with two ones, we see that the determinant of $A$ is the sum or difference of the determinants of two matrices with determinant in $\{-1, 0, 1\}$.
In either case, $|\det(A')| \le 2$.
But $|\det(A')| = |\det(A)|$, so this contradicts the choice of $A$.
\end{proof}

\begin{lemma} \label{A'_r}
For each integer $r \ge 2$, the matrix $A'_r$ is $2$-modular.
\end{lemma}
\begin{proof}
For each integer $r \ge 3$, let $H_r = [I_{r} | D_{r} | v]$, where $v = [1, -1, -1, 0, \dots, 0]^T$.
Then $A'_r$ can be obtained from $H_{r+1}$ by pivoting on the first entry of the last column, deleting the first row and last column, and then deleting two columns.
This implies that each $r \times r$ submatrix of $A'_r$ has the same determinant as an $(r+1) \times (r+1)$ submatrix of $H_{r+1}$, so it suffices to show that $H_r$ is $2$-modular for each $r \ge 3$.

Let $A = [A' | v]$ be an $r \times r$ submatrix of $H_r$.
Let $A_1 = [A' | e_1 - e_2]$ and $A_2 = [A' | e_3]$, where $e_i$ is the $i^{th}$ unit vector.
Then $|\det(A_1)| \le 1$ and $|\det(A_2)| \le 1$, since each is a submatrix of $[I_r | D_r]$.
Since the determinant is linear in the last column and $v = (e_1 - e_2) - e_3$, we have $\det(A) = \det(A_1) - \det(A_2)$, so $|\det(A)| \le 2$.
\end{proof}

The vector matroids of the matrices $A_r$ and $A'_r$ are both simplifications of elementary projections of $M(K_{r+2})$. 
For each integer $r \ge 2$, the matroid $T_r$ is the simplification of the matroid obtained by projecting $M(K_{r+2})$ by an element freely added to a $3$-point line.
The matrix $A_r$ is a well-known representation of $T_r$ \cite{OVW}, so $T_r$ is a $2$-modular matroid by Lemma~\ref{A_r}.
Note that $|T_r| = |M(K_{r+2})| - 2 = {r + 2 \choose 2} - 2$.
We now define another single-element extension of $M(K_{r+2})$, by its corresponding modular cut.

% the modular cut

\begin{proposition} \label{T'_r}
Let $L_1$ and $L_2$ be $2$-point lines of $M(K_{r+2})$ so that $L_1 \cup L_2$ is a $4$-element circuit. 
Then the set of flats of $M(K_{r+2})$ that contain $L_1$ or $L_2$ is a proper modular cut. 
Moreover, the corresponding single-element extension of $M(K_{r+2})$ is represented over both $\bR$ and $\GF(3)$ by the matrix $[ I_{r+1} | D_{r+1} | v ]$, where $v = [1, -1, -1, 0, \dots, 0]^T$.
\end{proposition}
\begin{proof}
Note that the matrix $H_{r+1} = [I_{r+1} | D_{r+1} | v]$ is a submatrix of $A'_{r+1}$, and is thus $2$-modular, by Lemma~\ref{A'_r}.
Therefore, the vector matroids of $H_{r+1}$ over $\bR$ and $\GF(3)$ are isomorphic; let $M$ denote this matroid.
In the matroid $M$, let $e$ label the last column of $H_{r+1}$.
Then $M\del e$ is $M(K_{r+2})$.
Let $\cF$ be the set of flats of $M\del e$ that span $e$ in $M$.
Note that $\cF$ contains no $1$-element flat, because $M$ is simple.
Let $a,b,c,d$ be the labels of the columns $[1, -1, 0, \dots, 0]^T$, $[0, 1, 0, \dots, 0]^T$, $[0, 0, 1, 0, \dots, 0]^T$, and $[1, 0, -1, 0, \dots, 0]^T$, respectively, of $H_{r+1}$.
Then $\{a, c, e\}$, $\{b, d, e\}$, and $\{a,b,c,d\}$ are circuits of $M$, and so $\{a,c\}$ and $\{b,d\}$ are flats of $M(K_{r+2})$ in $\cF$.
Let $F_0$ denote the $6$-element flat of $M(K_{r+2})$ spanned by $\{a,b,c,d\}$; then $M(K_{r+2})|F_0 \cong M(K_4)$.
It is straightforward to check that $\{a,c,e\}$ and $\{b,d,e\}$ are the only long lines of $M|(F_0 \cup \{e\})$ that contain $e$.

Suppose that there is a flat $F \in \cF$ that does not contain $\{a,c\}$ or $\{b,d\}$.
Since $F \in \cF$ and $F_0 \in \cF$, the flat $L = F \cap F_0$ is in $\cF$ as well, by property (ii) of modular cuts.
This flat $L$ has rank less than three, since it does not contain $\{a,c\}$ or $\{b,d\}$.
It has size at least two, since $M$ is simple.
But $\{a,c,e\}$ and $\{b,d,e\}$ are the only long lines of $M|(F_0 \cup \{e\})$ that contain $e$, which contradicts the existence of $F$.
Thus, $\cF$ is precisely the set of flats that contain $\{a,c\}$ or $\{b,d\}$.
Since $\cF$ is proper, the proposition holds.
\end{proof}

% define T'_r

For each integer $r \ge 2$, let $T'_r$ denote the simplification of the matroid obtained by projecting $M(K_{r+2})$ by an element corresponding to the modular cut of Proposition~\ref{T'_r}.
It is not hard to see that the matrix $A'_r$ can be obtained from the real matrix $[I_{r+1} | D_{r+1} | v]$ by pivoting on the first entry of the last column, deleting the first row and last column, and then deleting one column from each of the two parallel pairs.
Thus, $A'_r$ represents $T'_r$ over $\bR$, and so $T'_r$ is a $2$-modular matroid by Lemma~\ref{A'_r}.
Note that $|T'_r| = |M(K_{r+2})| - 2 = {r + 2 \choose 2} - 2$.
The matroids $T_r$ and $T'_r$ are likely the unique simple rank-$r$ $2$-modular matroids of size ${r + 2 \choose 2} - 2$ for $r \ge 5$, but this seems difficult to prove.
\end{section}

% excluded minors

\begin{section}{Duality and Excluded Minors} \label{duality}
In this section, we find several excluded minors for $\cM_2$, show that $\overline \cM_2 = \cM_2$, and present a proof showing that $\cM_{\Delta}$ is closed under duality for each positive integer $\Delta$.
We make use of the following result.

\begin{lemma} \label{pivot}
Let $A$ be a $2$-modular representation of a rank-$r$ matroid $M$, and let $e$ be a non-loop of $M$ that labels the first column of $A$.
Then there is a $2$-modular matrix $A'$ of the form $[I_r | X]$ whose first column is labeled by $e$ such that $A'$ is obtained from $A$ by elementary row operations, column swaps, and dividing rows by two.
\end{lemma}
\begin{proof}
We may assume that $e$ labels the first column of $A$, that $A$ has $r$ rows, and that $r \ge 2$.
Let $B$ be a basis of $M$ containing $e$, and let $A_B$ be the corresponding submatrix of $A$. 
By performing elementary row operations, we may transform $A$ into another $2$-modular matrix in which  $A_B$ is in Hermite normal form.
If each entry on the diagonal of $A_B$ is $1$, then the lemma holds.
Otherwise, since $A$ is $2$-modular, exactly one entry on the diagonal of $A_B$ is equal to $2$ and all the rest are equal to $1$.
We may assume that the row of $A$ with the entry $2$ on the diagonal of $A_B$ has an entry in $\{-1, 1\}$, or else we may divide this row by $2$.
By pivoting on that entry and swapping columns, we obtain a $2$-modular matrix of the form $[I_r | X]$, where the first column is labeled by $e$.
\end{proof}

Since row operations do not change the vector matroid, Lemma~\ref{pivot} implies that $\cM_2$ is closed under duality (see~\cite[Theorem 2.2.8]{Oxley}).
Also, since every $\Delta$-modular matrix of the form $[I_r | X]$ is in fact totally $\Delta$-modular, Lemma~\ref{pivot} implies that $\cM_2 = \overline \cM_2$.
Unfortunately, the analogue of Lemma~\ref{pivot} is false for $\Delta$-modular matrices with $\Delta \ge 3$.
For example, the matrix 
$$\begin{bmatrix}
1 & 1 & 1  \\
0 & 2 & 3  \\
\end{bmatrix}$$
with the first column labeled by $e$ is $3$-modular, but is not row-equivalent to a matrix with an $I_2$-submatrix that uses the first column.
This is the main difficulty in showing that $\overline \cM_{\Delta} = \cM_{\Delta}$ for $\Delta \ge 3$.

% duality

Next we present a proof showing that $\cM_{\Delta}$ is closed under duality.
This follows from a result of D'Adderio and Moci concerning arithmetic matroids \cite[Theorem 2.2]{DM}, but we present an elementary proof due to Marcel Celaya (private communication).
We also thank the anonymous referee for informing us that this result is in fact known.

We first develop some notation. 
Let $A \in \bZ^{m \times n}$ be a matrix with columns indexed by a set $E$.
For every $X \subseteq E$, we write $A_X$ for the submatrix of $A$ consisting of the columns indexed by $X$.
If $A$ has full row-rank, then we write $\gcd(A)$ for the greatest common divisor of the determinants of the $m \times m$ submatrices of $A$.
We write $[n]$ for $\{1,2,\dots, n\}$.

\begin{theorem}[D'Adderio, Moci] \label{duality}
For each positive integer $\Delta$, the class $\cM_{\Delta}$ is closed under duality.
\end{theorem}
\begin{proof}
Let $M$ be a $\Delta$-modular matroid with ground set $[n]$, and let $A \in \bZ^{m \times n}$ be a $\Delta$-modular representation of $M$ with full row rank.
We first show that there is a matrix $C \in \bZ^{(n-m) \times n}$ with $\gcd(C) = 1$ whose rows form a basis of $\ker(A)$.
Let $C_1 \in \bZ^{(n-m) \times n}$ be a matrix whose rows form a basis of $\ker(A)$.
Let $P \in \bZ^{(n - m) \times (n - m)}$ and $Q \in \bZ^{n \times n}$ be unimodular matrices so that $PC_1Q$ is the Smith normal form of $C_1$.
Then $PC_1Q = [D|0]$, where $D \in \bZ^{(n - m) \times (n - m)}$ is an invertible diagonal matrix.
Then $D^{-1}PC_1Q = [I|0]$ is an integer matrix with $\gcd(D^{-1}PC_1Q) = 1$.
Since $Q$ is unimodular, so is $Q^{-1}$, and therefore $(D^{-1}PC_1Q)Q^{-1} = D^{-1}PC_1$ is an integer matrix.
Clearly the Smith normal form of $D^{-1}PC_1$ is $D^{-1}PC_1Q = [I|0]$. 
Since a matrix $B$ with full row-rank has $\gcd(B) = 1$ if and only if the Smith normal form of $B$ is $[I|0]$, it follows that $\gcd(D^{-1}PC_1) = 1$.
%Since $Q^{-1}$ is unimodular, the matrix $(D^{-1}PC_1Q)Q^{-1} = D^{-1}PC_1$ is also an integer matrix with $\gcd(D^{-1}PC_1) = 1$.
Therefore, $D^{-1}PC_1$ is an integer matrix row-equivalent to $C_1$ with $\gcd(D^{-1}PC_1) = 1$, so we may take $C = D^{-1}PC_1$.

We will show that $C$ is a $\Delta$-modular representation of $M^*$.
For each set $X \subseteq [n]$, we write $\bar X$ for $[n] - X$.
Note that $A$ and $C$ both have columns indexed by $[n]$, and that $B \subseteq [n]$ is a basis of $M$ if and only if $\bar B$ is a basis of $M^*$.
By rearranging columns, we may assume that $B = [m]$ is a basis of $M$.
Then $A_B^{-1}A = [I_m | A_B^{-1}A_{\bar B}]$.
Since $A_B^{-1}A$ is row-equivalent to $A$, it represents $M$ over $\mathbb Q$.
Let $C' = [(A_B^{-1}A_{\bar B})^T| -I_{n - m}]$.
Then by Theorem 2.2.8 in \cite{Oxley}, $C'$ represents $M^*$ over $\mathbb Q$.
Also, the rows of $C'$ form a basis of $\ker(A_B^{-1}A) = \ker(A)$, so $C'$ is row-equivalent to $C$, and thus $C$ represents $M^*$ as well.
Moreover, since $C'$ and $C$ are row-equivalent, there exists an invertible rational $(n - m) \times (n - m)$ matrix $Q$ such that $C' = QC$.
Then each set $B_1 \subseteq E(M)$ of size $m$ satisfies $\det(C'_{\bar B_1}) = \det(Q)\det(C_{\bar B_1})$.

Let $B_0$ be a basis of $M$ obtained from $B$ by a basis exchange, so $B_0 = (B \cup \{j\}) - \{i\}$ for some $i \in B$ and $j \notin B$.
Then the matrix $C'_{\bar B_0}$ is obtained from $-I_{n - m}$ by replacing the $(j - m)^{th}$ column with the $i^{th}$ column of $(A_B^{-1}A_{\bar B})^T$,  so 
\begin{align}
\det(C'_{\bar B_0}) &=  (-1)^{n - m - 1}(A_B^{-1}A_{\bar B})^T_{j-m,i} \\
&= (-1)^{n - m - 1}(A_B^{-1}A_{\bar B})_{i,j-m} \\
&=  (-1)^{n - m - 1}(A_B^{-1}A)_{i,j}.
\end{align}
Furthermore, the system $(A_B) x = A_{\{j\}}$ has unique solution $A_B^{-1}A_{\{j\}} = (A_B^{-1}A)_{\{j\}}$, so by Cramer's rule, the $i^{th}$ entry of the solution $(A_B^{-1}A)_{\{j\}}$ is 
\begin{align*}
(A_B^{-1}A)_{i,j} = \frac{\det(A_{B_0})}{\det(A_B)}.
\end{align*}
Combined with (3), this implies that 
\begin{align}
\det(C'_{\bar B_0 }) = (-1)^{n - m - 1} \frac{\det(A_{B_0})}{\det(A_B)}.
\end{align}
Since $C' = QC$ and $\det(C'_{\bar B}) = (-1)^{n - m}$, (4) implies that 
\begin{align}
\frac{\det(C_{\bar B_0 })}{\det(C_{\bar B})} = \frac{\det(C'_{\bar B_0})}{\det(C'_{\bar B})} = -\frac{\det(A_{B_0})}{\det(A_B)}.
\end{align}
Since every basis of $M$ can be obtained from any other basis by a sequence of basis exchanges, (5) implies that each pair $(B_1, B_2)$ of bases of $M$ satisfies 
 \begin{align}
 \frac{\det(C_{\bar B_1})}{\det(C_{\bar B_2})} =  \pm\frac{\det(A_{B_1})}{\det(A_{B_2})}.
 \end{align}

Suppose that some basis $B_1$ of $M$ satisfies $\det(C_{\bar B_1}) =c \cdot \det(A_{B_1})$ for some $c \in \mathbb Q$ with $|c| > 1$.
Then (6) implies that each other basis $B_2$ of $M$ satisfies $\det(C_{\bar B_2}) = \pm c \cdot \det(A_{B_2})$.
But then $c \cdot \gcd(A)$ is an integer with absolute value greater than $1$ which divides $\gcd(C)$, a contradiction.
Thus, each basis $B$ of $M$ satisfies $|\det(C_{\bar B})| \le |\det(A_B)| \le \Delta$, so $C$ is $\Delta$-modular.
\setcounter{equation}{0}
\end{proof}

%We remark that it is unknown whether $\overline \cM_{\Delta}$ is closed under duality, but this would be implied by Theorem \ref{duality} if $\overline \cM_{\Delta} = \cM_{\Delta}$.

We now turn our attention back to $\cM_2$.
To find excluded minors for $\cM_2$, we can use the fact that every matroid in $\cM_2$ is dyadic.
The excluded minors for the class of dyadic matroids include $U_{2,5}$, $F_7$, $\AG(2,3)\del e$, $\Delta(\AG(2,3)\del e)$, $T_8$, and their duals.
Here $\Delta(\AG(2,3)\del e)$ is obtained from $\AG(2,3)\del e$ by a $\Delta$-$Y$ exchange, while $T_8$ is a self-dual ternary spike with the tip deleted \cite[p.~649]{Oxley}.
In fact, each of these matroids is an excluded minor for $\cM_2$.
%; we do not need this fact for the proof of Theorem~\ref{main matroids}.
The proof, which we omit, amounts to finding a $2$-modular representation for each single-element deletion or contraction of each of these matroids.
Also, since the ternary Reid geometry $R_9$ has $\AG(2,3)\del e$ as a restriction, $R_9$ is not in $\cM_2$.
There are two other known excluded minors for the class of dyadic matroids, namely $N_1$ and $N_2$, and the question of whether this list is complete is a long-standing open problem (see \cite[Problem 14.7.11]{Oxley}).

\begin{problem} \label{dyadic}
Determine if $$\{U_{2,5}, U_{3,5}, F_7, F_7^*, \AG(2,3)\del e, (\AG(2,3)\del e)^*, \Delta(\AG(2,3)\del e), T_8, N_1, N_2\}$$ is the complete list of excluded minors for the class of dyadic matroids.
\end{problem}

The matroid $N_2$ may be an excluded minor for $\cM_2$, but $N_1$ is not, as we shall see shortly.
To find excluded minors for $\cM_2$ that are dyadic, we use Lemma~\ref{pivot}.

\begin{proposition} \label{U}
$\U$ is an excluded minor for $\cM_2$.
\end{proposition}
\begin{proof}
Suppose $A$ is a $2$-modular representation of $\U$ with four rows.
By Lemma~\ref{pivot}, we may assume that $A$ has an $I_4$-submatrix.
Since $U_{2,4}$ is not regular, this implies that $A$ has a $4 \times 4$ block-diagonal submatrix where the determinant of each nonzero block has absolute value at least two.
But then the determinant of this submatrix has absolute value at least four, a contradiction, so $\U \notin \cM_2$.
Each single-element contraction of $\U$ is isomorphic to $U_{1,3} \oplus U_{2,4}$, and it is easy to see that this matroid is in $\cM_2$.
Since $\U$ is self-dual, this implies that every minor of $\U$ is in $\cM_2$.
\end{proof}

The matroid $\U$ is signed-graphic, which means that it has a representation over $\GF(3)$ so that each column has at most two nonzero entries.
In fact, $\cM_2$ has at least two other signed-graphic excluded minors, represented over $\bR$ by the following matrices:
$$\begin{bmatrix}
1 & 0 & 1 & 1 & 0 & 0 & 1 & 0 \\
0 & 1 & 1 & -1 & 0 & 0 & 0 & 1 \\
0 & 0 & 0 & 0 & 1 & 1 & 0 & -1 \\
0 & 0 & 0 & 0 & 1 & -1 & -1 & 0
\end{bmatrix}, \begin{bmatrix}
1 & 0 & 1 & 1 & 0 & 0 & 0 & 0 \\
0 & 1 & 1 & -1 & 0 & 0 & 0 & 1 \\
0 & 0 & 0 & 0 & 1 & 1 & 0 & -1 \\
0 & 0 & 0 & 0 & 1 & -1 & 1 & 0
\end{bmatrix}$$
We write $U_8$ for the vector matroid of the first matrix, and $U_8'$ for the vector matroid of the second matrix.
Both matroids are self-dual.
One can prove that they are excluded minors for $\cM_2$ using a similar proof to that of Proposition~\ref{U}.
The excluded minor $N_1$ for the dyadic matroids has a $U_8$-minor, but $N_2$ does not have a minor isomorphic to $\U$, $U_8$, or $U'_8$, and thus it may be an excluded minor for $\cM_2$.

Finding the excluded minors for $\cM_{\Delta}$ likely becomes more difficult as $\Delta$ increases.
Indeed, we have not yet solved the following.

\begin{problem} \label{rank-2}
Find the unique rank-$2$ excluded minor for $\cM_{\Delta}$.
\end{problem}

Since $\cM_{\Delta}$ is closed under parallel extension and adding loops, each excluded minor for $\cM_{\Delta}$ is simple.
Since every rank-$2$ simple matroid is uniform, this implies that $\cM_{\Delta}$ has a unique rank-$2$ excluded minor, which must be uniform.
If $p$ is the smallest prime greater than $\Delta$, then $U_{2, p+2} \notin \cM_{\Delta}$, because every matroid in $\cM_{\Delta}$ is $\GF(p)$-representable.
Also, it is not hard to see that $U_{2, \Delta + 2} \in \cM_{\Delta}$ for each $\Delta \ge 1$, by taking the representation with two rows, two unit columns, and all columns of the form $[1, m]^T$, where $1 \le m \le \Delta$.
This implies that if $\Delta + 1$ is prime, then $U_{2, \Delta + 3}$ is an excluded minor for $\cM_{\Delta}$.
It is natural to conjecture that $U_{2, p + 2}$ is always an excluded minor for $\cM_{\Delta}$, where $p$ is the smallest prime greater than $\Delta$.
However, it appears (through Sage code) that the rank-$2$ excluded minor for $\cM_7$ is $U_{2, 11}$, even though nine is not prime.
Thus, it is unclear what the correct answer to Problem \ref{rank-2} may be as a function of $\Delta$.
\end{section}

% start of main proof

\begin{section}{Single-Element Extensions} \label{1-extensions}
The remainder of this paper is devoted to the proof of Theorem~\ref{main matroids}.
In this section, we show that, up to isomorphism, there are only two rank-$r$ non-trivial single-element extensions of a clique that have no minor isomorphic to $U_{2,5}$, $F_7$, or $R_9$. 
We need the following straightforward lemma to identify $R_9$.

\begin{lemma} \label{R}
Let $M$ be a simple rank-$3$ matroid on nine elements consisting of three long lines through a common point. 
Then either $M$ has a $U_{2,5}$-minor, or 
$M \cong R_9$.
\end{lemma}
\begin{proof}
Suppose that $M$ has no $U_{2,5}$-minor, and let $x$ be a point that is on three long lines.
Since $|M| = 9$, two of these lines have four points, and one has three points. 
Let $L_1$ and $L_2$ be the $4$-point lines through $x$, and let $\{x,y,z\}$ be the $3$-point line through $x$.
Since $M$ has no $U_{2,5}$-minor, each element of $L_1 - \{x\}$ is on a long line with $y$ and a long line with $z$.
Thus, $y$ and $z$ are both on four long lines of $M$, so $M \cong R_9$.
\end{proof}

We now show that, up to isomorphism, there are only two rank-$r$ non-trivial single-element extensions of a clique that have no minor isomorphic to $U_{2,5}$, $F_7$, or $R_9$.

\begin{proposition} \label{projections}
Let $M$ be a simple matroid of rank at least three with a set $X$ so that $M|X \cong M(K_{r(M)+1})$, and let $e \in E(M) - X$.
If $M$ has no minor isomorphic to $U_{2,5}$, $F_7$, or $R_9$, then either
\begin{enumerate}[$(a)$]
\item the extension of $M|X$ by $e$ corresponds to the modular cut generated by a $3$-point line, or

\item the extension of $M|X$ by $e$ corresponds to the modular cut generated by a pair of $2$-point lines whose union is a circuit.
\end{enumerate}
\end{proposition}
\begin{proof}
Let $\cL$ denote the set of lines of $M|X$ that span $e$.
Since $M|X \cong M(K_{r(M) + 1})$, each line in $\cL$ has size at most three.
The $3$-point lines in $\cL$ correspond to triangles of $K_{r(M) + 1}$, and the $2$-point lines in $\cL$ correspond to $2$-edge matchings of $K_{r(M) + 1}$.

\begin{claim} \label{simple}
$\cL$ is non-empty.
\end{claim}
\begin{proof}
Assume that $\cL$ is empty.
Then $(M/e)|X$ is simple.
Let $M_1$ be a minimal simple minor of $M$ of rank at least three so that 
\begin{itemize}
\item $e \in E(M_1)$, 

\item there is a set $X_1\subseteq X$ so that $M_1|X_1 \cong M(K_{r(M_1) + 1})$, and 

\item $(M_1/e)|X_1$ is simple.
\end{itemize}

Since $M$ has no $U_{2,5}$-minor and $(M_1/e)|X_1$ is simple, each plane of $M_1|X_1$ that spans $e$ has size at most four.
Thus, $r(M_1) \ge 4$.
By minimality, each element of $X_1$ is in a plane of $M_1|X_1$ that spans $e$.
Each plane of $M_1|X_1$ has size $3$ or $4$, corresponding to a $3$-edge matching, or a triangle and disjoint edge.
Since each plane of $M_1|X_1$ is spanned by an $M(K_m)$-restriction of $M_1|X_1$ for some $m \le 6$, the minimality of $M_1$ implies that $r(M_1) \le 5$.

Suppose $r(M_1) = 4$.
Then each element of $X_1$ is in a $4$-element plane of $M_1|X_1$ that spans $e$. 
Since $|X_1| = 10$, there is some $x \in X_1$ that is in at least two such planes.
Each element of $X_1$ is on at least three long lines of $M_1|X_1$, and is thus on at least three long lines of $(M_1/e)|X_1$, or else $M_1/e$ has a $U_{2,5}$-restriction, since $(M_1/e)|X_1$ is simple.
But $x$ is on two $4$-point lines in $M_1/e$, so Lemma~\ref{R} implies that $M_1$ has a minor isomorphic to $U_{2,5}$ or $R_9$, a contradiction.
Thus, $r(M_1) = 5$.

By the minimality of $M_1$, each plane of $M_1|X_1$ that spans $e$ has size three.
Moreover, each element $x\in X_1$ is in at least two planes of $M_1|X_1$ that span $e$; otherwise $e$ is spanned by a unique line of $(M_1/x)|(X_1 - \{x\})$, and this line has two points, so $M_1/\{x,e\}$ has a $U_{2,5}$-restriction.
We may assume that the ground set of $M_1|X_1$ is the edge set of the complete graph with vertex set $\{0,1,2,3,4,5\}$.
Let $x = 01$ and choose $y = 02$.
We will show that $M_1/\{e,y\}$ has at least three long lines through $x$, two of which have four points, and obtain a contradiction using Lemma~\ref{R}.
Let $Z$ be the union of the three $3$-point lines of $M_1|X_1$ that contain $x$ but not $y$.

Note that $Z$ does not span $e$ or $y$ in $M_1$.
Suppose $(M_1/\{e,y\})|Z$ is not simple.
Then there is a circuit $\{x_1, x_2, e, y\}$ of $M_1$ where $x_1$ and $x_2$ are in $Z - \{x\}$.
Therefore, $\{y, x_1, x_2\}$ is not a $3$-edge matching in the $K_6$ corresponding to $M_1|X_1$, contradicting the fact that each plane of $M_1|X_1$ spanning $e$ has size three.
Thus, $(M_1/\{e,y\})|Z$ is simple, so there are at least three long lines of $M_1/\{e,y\}$ that contain $x$.

There is no plane of $M_1|X_1$ that spans $e$ and contains both $x$ and $y$, since such a plane would have at least four elements.
Since $x$ is in at least two planes of $M_1|X_1$ that span $e$, this implies that $x$ is in at least two planes of $(M_1/y)|(X_1 - \{y\})$ that span $e$ in $M_1/y$.
Moreover, these planes each have at least four points, since $\si((M_1/y)|(X_1 - \{y\})) \cong M(K_5)$ and so has no $3$-point plane.
Thus, in $M_1/\{y,e\}$, the element $x$ is on at least two $4$-point lines.
Since $x$ is on at least three long lines, Lemma~\ref{R} implies that $M_1/\{e,y\}$ has an $R_9$-restriction or a $U_{2,5}$-minor, a contradiction.
\end{proof}

% now \cL is non-empty

Since $\cL$ is non-empty, the sets in $\cL$ are pairwise disjoint and pairwise coplanar.
To finish the proof, we will use the following two facts:
\begin{enumerate}[(i)]
\item If $(F, F')$ is a modular pair of flats of $M|X$ such that $F$ and $F'$ both span $e$, then the flat $F\cap F'$ spans $e$ as well.

\item If $F$ and $F'$ are flats of $M|X$, and $M|F'$ is isomorphic to a clique, then $(F, F')$ is a modular pair.
\end{enumerate}

First assume that $\cL$ contains a $3$-point line $L$.
If $F$ is a flat of $M|X$ that spans $e$, then $F \cap L$ is a flat of $M|X$ that spans $e$, by (i) and (ii).
Since $M$ is simple, $r(F \cap L) \ne 1$, so $F$ contains $L$.
Thus, $(a)$ holds, so we may assume that each line in $\cL$ has size two.
Then each such line corresponds to a $2$-edge matching in $K_{r(M) + 1}$.
Since the lines in $\cL$ are pairwise coplanar, the union of two such lines corresponds to a $4$-cycle of $K_{r(M) + 1}$.
Thus, $|\cL| \le 3$.
Moreover, if $\cL = \{L_1, L_2, L_3\}$, then $M|(L_1\cup L_2\cup L_3) \cong M(K_4)$.
As every element of $M|(L_1\cup L_2 \cup L_3 \cup \{e\})$ is on three long lines, $M|(L_1\cup L_2 \cup L_3 \cup \{e\}) \cong F_7$, a contradiction.
Thus, $|\cL| \le 2$.

Let $Z \subseteq X$ so that $M|Z \cong M(K_4)$ and $Z$ contains each line in $\cL$.
If $|\cL| = 1$, then $(M/e)|Z \cong U_{2,5}$, a contradiction.
Therefore $|\cL| = 2$.
Let $L_1$ and $L_2$ be the two lines in $\cL$.
If $F$ is a flat of $M|X$ that spans $e$, then $F \cap Z$ is a flat of $M|X$ that spans $e$, by (i) and (ii).
Since $M$ is simple, the flat $F \cap Z$ has rank at least two.
If $F$ does not contain $L_1$ or $L_2$, this implies that $F \cap Z$ is a line of $M|X$ other than $L_1$ or $L_2$.
This gives a contradiction as either $\cL$ contains a $3$-point line, or $|\cL| > 2$.
We conclude that $F$ contains $L_1$ or $L_2$, and so $(b)$ holds.
\end{proof}

Let $M$ be a matroid with a set $X$ so that $M|X \cong M(K_{r(M)+1})$.
We say that an element $e \in E(M) - X$ is \emph{type-$(a)$} if $e$ is freely added to a $3$-point line of $M|X$, and is \emph{type-$(b)$} if the extension of $M|X$ by $e$ corresponds to the modular cut generated by a pair of $2$-point lines of $M|X$ whose union is a circuit.
Proposition~\ref{projections} says that if $M$ is simple and has no minor isomorphic to $U_{2,5}$, $F_7$, or $R_9$, then every element in $E(M) - X$ is type-$(a)$ or type-$(b)$.
Also note that $e$ is type-$(a)$ if and only if $\si((M/e)|X) \cong T_{r(M) - 1}$, and $e$ is type-$(b)$ if and only if $\si((M/e)|X) \cong T'_{r(M) - 1}$.
\end{section}

% 2-element extensions

\begin{section}{Two-Element Extensions} \label{2-extensions}
We now analyze the structure of a $2$-element extension of a clique, under the additional assumption that $U_{2,4} \oplus U_{2,4}$ is excluded as a minor.
In a matroid $M$, for sets $A, B \subseteq E(M)$, we write $\sqcap(A,B)$ for $r_M(A) + r_M(B) - r_M(A\cup B)$.

\begin{lemma} \label{2 and 3}
Let $M$ be a simple matroid of rank at least five with no minor in $\{U_{2,5}, F_7, R_9, \U\}$, and with a set $X$ so that $M|X \cong M(K_{r(M)+1})$.
If $e, f\in E(M) - X$ so that $e$ is type-$(a)$ and $f$ is type-$(b)$, then there is a set $Z \subseteq X$ so that $M|Z \cong M(K_4)$ and $e,f \in \cl_M(Z)$. 
\end{lemma}
\begin{proof}
Since $e$ is a type-$(a)$ element, there is a $3$-point line $L_1$ of $M|X$ so that $e$ is freely added to $L_1$.
By Proposition~\ref{T'_r}, since $f$ is a type-$(b)$ element, there are $2$-point lines $L_2$ and $L_2'$ of $M|X$ so that $L_2 \cup L_2'$ is a $4$-element circuit, and a flat $F$ of $M|X$ spans $f$ if and only if $F$ contains $L_2$ or $L_2'$.

We may assume that $\cl_M(L_2 \cup L_2')$ does not contain $e$ otherwise the lemma holds.
Thus $\sqcap(L_2 \cup L_2', L_1) \ne 2$.
If $\sqcap(L_2 \cup L_2', L_1) = 0$, then $L_1 \cup \{e\}$ and $\cl_{M|X}(L_2 \cup L_2')$ are skew rank-$2$ sets in $M/f$ each with four points, so $M$ has $\U$ as a minor, a contradiction.
We deduce that $\sqcap(L_2 \cup L_2', L_1) = 1$.
Hence there is some $Z \subseteq X$ such that $L_1\cup L_2 \cup L_2' \subseteq Z$ and $M|Z \cong M(K_5)$. 
Let $Z \subseteq Y \subseteq X$ so that $M|Y \cong M(K_6)$. 
We may assume that the ground set of $M|Y$ is the edge set of the complete graph with vertex set $\{0,1,2,3,4,5\}$, and that the ground set of $M|Z$ is the edge set of the complete graph with vertex set $\{0,1,2,3,4\}$.
As $\sqcap(L_2 \cup L_2', L_1) = 1$, we may assume that $L_1 = \{01, 02, 12\}$, and $L_2 = \{13, 24\}$.

Suppose that $\{z, e, f\}$ is a circuit of $M$ for some $z \in Z$. 
Then either $z$ meets precisely one of $0$, $1$, or $2$, or $z = 34$.
In the first case, $L_1 \cup \{z\}$ spans an $M(K_4)$-restriction of $M|X$, so $f$ is spanned by a plane of $M|X$ that does not contain $L_2$ or $L_2'$, a contradiction.
If $z = 34$, then $e \in \cl_M(L_2 \cup L_2')$, a contradiction.
Thus, no such element $z$ exists.

% L_2' does not intersect

Assume that $L_2' = \{14, 23\}$.
As $e \notin \cl_M(L_2 \cup L_2')$, the elements $13, 14, 12, 34$ are on distinct lines of $M/e$ through $f$.
Moreover, each element of $(M/e)|(Z \cup \{f\})$ is on one of these four lines, or else $M/\{e,f\}$ has a $U_{2,5}$-restriction.
The sets $\{f, 13, 24\}$ and $\{f, 14, 23\}$ have rank two in $M/e$, and do not span any element in $\{03, 04, 34\}$, otherwise $e$ is spanned by a plane of $M|Z$ that does not contain $L_1$, a contradiction.
This means that each element in $\{03, 04, 34\}$ is spanned in $M/e$ by either $\{f, 12\}$ or $\{f, 34\}$.
Since $r_{M/e}(\{03, 04, 34\}) = 2$, this implies that $\{03, 04, 34\}$ is spanned in $M/e$ by $\{f, 34\}$, so $\{03, 04, 34, f\}$ is a $U_{2,4}$-restriction of $(M/e)|(Z \cup \{f\})$.
But $\{01, 03, 04, 05\}$ is a basis of $(M/e)|Y$, and $\{01, 05, 15, 25 \}$ is a $U_{2,4}$-restriction of $M/e$, since it spans $L_1$ in $M$.
Thus, $M/e$ has a $(\U)$-restriction, a contradiction.

% L_2' intersects

Now assume that $L_2' = \{12, 34 \}$.
Let $N = (M/e)|(Z \cup \{f\})$, and let $N' = N \del \{01, 02\}$.
Then $N'$ is a simple rank-$3$ matroid on nine elements. 
The sets $\{12, 03, 13, 23\}$ and $\{12, 04, 14, 24\}$ are lines of $N'$ because they span $e$ in $M$, and the set $\{12, 34, f\}$ is a line of $N'$ because it is a line of $M$ that does not span $e$. 
But then $N'$ has a $U_{2,5}$-minor or $N' \cong R_9$, by Lemma~\ref{R}, a contradiction.
\end{proof}

The following lemma imposes structure on a different $2$-element extension of a clique.

\begin{lemma} \label{2 and 2}
Let $M$ be a simple matroid of rank at least six with no minor in $\{U_{2,5}, F_7, R_9, \U\}$ and with a set $X$ so that $M|X \cong M(K_{r(M)+1})$.
Let $C_1$ and $C_2$ be distinct $4$-element circuits of $M|X$ so that, for each $i \in \{1,2\}$, there is an element $e_i \in E(M) - X$ spanned by both $2$-point lines of $M|X$ contained in $C_i$.
Then $|C_1 \cap C_2| = 2$, and there is a set $Z \subseteq X$ so that $M|Z \cong M(K_5)$ and $C_1 \cup C_2 \subseteq Z$.
\end{lemma}
\begin{proof}
We may assume that $E(M) = X \cup \{e_1, e_2\}$.
Note that $e_1 \ne e_2$, and that $e_1$ and $e_2$ are both type-$(b)$.
Let $x$ and $y$ be distinct elements in $\cl_{M|X}(C_1) - C_1$.
Then either $x$ or $y$ is not in $\cl_{M|X}(C_2)$, or $C_1$ and $C_2$ are contained in an $M(K_4)$-restriction of $M|X$ and the lemma follows.
Assume that $x \notin \cl_{M|X}(C_2)$.
The matroid $(M/x)|(X-\{x\})$ has a spanning-clique restriction, and $e_2$ is spanned by both $2$-point lines of $(M/x)|(X-\{x\})$ contained in $C_2$.
Moreover, the simplification of $(M/x)|(C_1 \cup \{y\})$ is a $3$-point line, $L$, that spans $e_1$ in $M/x$.
Then $L \cup C_2$ is contained in an $M(K_4)$-restriction of $M/x$, by Lemma~\ref{2 and 3}.
Since $C_1$ spans $x$ in $M$, this implies that there is a set $Z \subseteq X$ so that $M|Z \cong M(K_5)$ and $C_1 \cup C_2 \subseteq Z$. 

We now prove that $|C_1 \cap C_2| = 2$.
We may assume that the ground set of $M|Z$ is the edge set of the complete graph with vertex set $\{0,1,2,3,4\}$.
Assume that $|C_1 \cap C_2| \ne 2$.
Then $|C_1 \cap C_2| = 1$ since $K_5$ has no edge-disjoint $4$-cycle pairs.
Up to relabeling vertices, we may assume that $C_1 = \{01, 12, 23, 03\}$ and $C_2 = \{01, 14, 42, 02\}$.

Suppose there is an element $z \in Z$ for which $\{z, e_1, e_2\}$ is a circuit of $M$.
Assume $z$ is in $\cl_M(C_1)$ or $\cl_M(C_2)$. 
Then $C_1$ or $C_2$ spans both $e_1$ and $e_2$.
Since $(\cl_{M}(C_1), \cl_{M}(C_2))$ is a modular pair of flats of $M$ whose intersection contains $\{01, 12, 02\}$ and $e_1$ or $e_2$, we deduce that $e_1$ or $e_2$ is spanned by $\{01, 12, 02\}$, a contradiction.
We conclude that $z \notin \cl_M(C_1) \cup \cl_M(C_2)$, and so $z = 34$.
% so $C_1$ and $C_2$ are distinct $4$-element circuits of a simplification of $(M/z)|(Z - \{z\})$ that span $e_1$ and $e_2$, respectively.
The set $\cl_M(\{12, 03, z\})$ contains neither $\{01, 24\}$ nor $\{02, 14\}$, and thus $e_2 \notin\cl_M(\{12, 03, z\})$ because $e_2$ is type-$(b)$.
This implies that $e_2 \notin \cl_{M/z}(\{12, 03\})$. 
But $e_1 \in \cl_{M/z}(\{12, 03\})$ because $e_1 \in \cl_M(\{12, 03\})$, and so $e_1$ and $e_2$ are not parallel in $M/z$, a contradiction.

The set $\{01, 02, 03, 13\}$ is a $U_{2,4}$-restriction of $M/e_1$.
Since $01$ and $23$ are parallel in $M/e_1$ and $\{e_2, 01, 24\}$ is a line of $M$ and $\{e_1, e_2\}$ spans no element of $X$, the set $\{01, 24, 34, e_2\}$ is a $U_{2,4}$-restriction of $M/e_1$.
Since $\{01, 04, 14\}$ is a $U_{2,3}$-restriction of $M/e_1$ and $M/e_1 \del \{12, 23\}$ is simple, Lemma~\ref{R} implies that $M/e_1$ restricted to the union of the three sets $\{01, 02, 03, 13\}$, $\{01, 24, 34, e_2\}$, and $\{01, 04, 14\}$ is either isomorphic to $R_9$, or has a $U_{2,5}$-minor, a contradiction.
\end{proof}
\end{section}

% k-element extensions

\begin{section}{$k$-Element Extensions} \label{k-extensions}
In this section, we place structure on $k$-element extensions of a clique with $k \ge 3$, in three cases.
We say that a point $x$ of a simple matroid $M$ is \emph{special} if it is on at least two $4$-point lines, is a tip of at least two spike restrictions of $M$, or is a tip of a spike restriction of $M$ and is on a $4$-point line.

% size-3 case

\begin{lemma} \label{3 case}
Let $M$ be a simple matroid with no minor in $\{U_{2,5}, F_7, R_9\}$, and with a set $X$ so that $M|X \cong M(K_{r(M) + 1})$.
If $|E(M) - X| \le 3$, then $M$ has at most $21$ special points.
\end{lemma}
\begin{proof}
If $r(M) \le 2$, then $|E(M)| \le |X| + |E(M) - X| \le 6$ and therefore $M$ has at most $6$ special points, so we may assume that $r(M) \ge 3$.
Since $|E(M) - X| \le 3$ and $M$ has no $U_{2,5}$-restriction, the points in $E(M) - X$ span at most three lines, each of which spans at most two points in $X$.
By Proposition~\ref{projections}, there are at most $12$ elements in $X$ that are on a line of $M|X$ that spans an element of $E(M) - X$.
Thus, there are at most $18$ elements in $X$ that are on a long line of $M$ that is not a long line of $M|X$.
Because no spike is graphic, $M|X$ has no spike restriction. 
Therefore, at most $18$ elements in $X$ are a tip of a spike or on a $4$-point line.
Since $|E(M) - X| \le 3$, it follows that $M$ has at most $21$ special points.
\end{proof}

% T_r case

We must work harder when $|E(M) - X|$ is not bounded.

\begin{lemma} \label{T_r case}
Let $M$ be a simple matroid with no minor in $\{U_{2,5}, F_7, R_9, \U\}$, and with a set $X$ so that $M|X \cong M(K_{r(M) + 1})$. 
If $|E(M) - X| \ge 4$ and each element in $E(M) - X$ is type-$(a)$, then $M$ has exactly one special point.
Moreover, this point is in every $3$-point line of $M|X$ which spans an element in $E(M) - X$.
\end{lemma}
\begin{proof}
We may assume that the ground set of $M|X$ is the edge set of the complete graph with vertex set $\{0,1,2, \dots, r(M)\}$.
Since each point in $E(M) - X$ is type-$(a)$, each is freely added to a $3$-point line of $M|X$.
Let $\cL$ denote the set of $3$-point lines of $M|X$ that span an element of $E(M) - X$.

\begin{claim} \label{x}
The lines in $\cL$ intersect in a common point.
\end{claim}
\begin{proof}
Suppose that the lines in $\cL$ do not intersect in a common point.
Since $M$ has no $(\U)$-minor, each pair of lines in $\cL$ has a common element.
Then since $|E(M) - X| \ge 4$, $M$ has a restriction $N$ that can be obtained from $M(K_4)$ by freely adding a point to each long line.
Let $z$ be an element of this $M(K_4)$-restriction.
Then $z$ is on at least two $4$-point lines of $N$.
Suppose $z$ is on three or more long lines of $N$. 
Then by Lemma \ref{R}, $N$ has a minor isomorphic to $R_9$ or $U_{2,5}$, a contradiction.
Therefore $z$ is on exactly two long lines of $N$.
But since $|N| = 10$ and $N$ has no $U_{2,5}$-restriction, this implies that $N/z \cong U_{2,5}$, a contradiction. 
\end{proof}

Without loss of generality, we may assume that the lines in $\cL$ intersect in the point $x = 01$.
Since $|\cL| \ge 4$, $x$ is special. 
We will show that $x$ is the unique special point of $M$.
Note that the points of $(M/x)|(X - \{x\})$ of size at least two form a basis of $\si((M/x)|(X - \{x\}))$; this implies that $E(M) - X$ is independent in $M$.

\begin{claim} \label{long lines}
Let $e_1, e_2 \in E(M) - X$, and, for each $i \in \{1,2\}$, let $L_i$ be a line in $\cL$ that spans $e_i$.
If $z \in X$ so that $\{z, e_1, e_2\}$ is a circuit of $M$, then $z$ is the unique element in $\cl_{M|X}(L_1 \cup L_2) - (L_1 \cup L_2)$.
Moreover, every $4$-point line of $M$ contains exactly three elements of $X$ including $x$.
\end{claim}
\begin{proof}
Certainly $z \in \cl_M(L_1 \cup L_2)$ since $\{z, e_1, e_2\}$ is a circuit.
If $z \in L_1$, then $L_1$ spans $e_2$.
But then $L_1$ is a flat of $M|X$ that spans $e_2$ but does not contain $L_2$, which contradicts Proposition~\ref{projections}(i).
Thus, $z \notin L_1$, and, by symmetry, $z \notin L_2$.
Since $|\cl_{M|X}(L_1 \cup L_2)| = 6$ and $|L_1 \cup L_2| = 5$, the element $z$ is unique.
As $E(M) - X$ is independent, it follows that each $4$-point line of $M$ contains exactly three elements of $X$ including $x$.
\end{proof}

By \ref{long lines}, if an element $y$ of $X$ is on a line in $\cL$, then each long line of $M$ through $y$ is spanned by a long line of $M|X$ through $y$.
As no spike is graphic, $y$ is not a tip of a spike restriction of $M$.
Moreover, if $y \ne x$, then $y$ is on exactly one $4$-point line of $M$, so $X$ contains at most one special point of $M$ that is on a line in $\cL$.

If an element $y$ of $X$ is not on a line in $\cL$, then $y$ is not on a $4$-point line, since $E(M) - X$ is independent.
By \ref{long lines}, the element $y$ is on at most one more long line in $M$ than in $M|X$, which implies that $y$ is a tip of at most one spike restriction of $M$.
Thus, $y$ is not a special point.
Therefore, $M$ has at most one special point in $X$, namely $x$.

Now, let $e \in E(M) - X$.
We will show that $e$ is not special.
We may assume that $e$ is freely added to the $3$-point line $L' = \{01, 02, 12\}$ of $M|X$.
By \ref{long lines}, each long line of $M$ through $e$ contains an element of $X$ that uses the vertex $2$.
Let $T$ be a transversal of the points of $M/e$ so that each element in $T$ is in $X$ and uses the vertex $2$; then $T$ is independent in $M|X$.
Also, $T$ contains at most one of $02$ and $12$, and so $T$ does not span $L'$ in $M|X$.
Since $e$ is freely added to $L'$, this implies that $T$ does not span $e$ in $M$.
Since $T$ is independent in $M$ and does not span $e$, the element $e$ is not a tip of a spike restriction of $M$.
Since $e$ is on exactly one $4$-point line of $M$, the element $e$ is not special.
Thus, $x$ is the unique special point of $M$, so the lemma holds.
\end{proof}

% T'_r case

We need an analogous lemma for a different case.

\begin{lemma} \label{T'_r case}
Let $M$ be a simple matroid with no minor in $\{U_{2,5}, F_7, R_9, \U\}$, and with a set $X$ so that $M|X \cong M(K_{r(M) + 1})$. 
If there is a $3$-point line $\{x,y,z\}$ of $M|X$ so that each element in $E(M) - X$ is either freely added to $\{x,y,z\}$, or is spanned by both $2$-point lines of $M|X$ in some $4$-element circuit of $M|X$ containing $x$ and $y$,
then $M$ has at most two special points.
\end{lemma}
\begin{proof}
Let $L = \{x,y,z\}$.
We may assume that the ground set of $M|X$ is the edge set of the complete graph with vertex set $\{0,1,2, \dots, r(M)\}$, and that $(x,y,z) = (01, 02, 12)$.
By Proposition~\ref{projections}, each element in $E(M) - X$ that is not freely added to $L$ is type-$(b)$. 
Let $\cC$ denote the set of circuits $C$ of $M|X$ for which both $2$-point lines of $M|X$ in $C$ span a common element of $E(M) - X$.
By hypothesis, each circuit in $\cC$ contains $x$ and $y$.

The points of $(M/x)|(X - \{x\})$ of size at least two form a basis of $\si((M/x)|(X - \{x\}))$ because they all use the composite vertex obtained by identifying $0$ and $1$.
This implies that $E(M) - X$ is independent in $M$.
Since $E(M) - X$ is independent, there is at most one $4$-point line of $M$, namely $\cl_M(L)$.

\begin{claim} \label{two circuits}
Let $e_1, e_2 \in E(M) - X$ so that neither is freely added to $L$, and, for each $i \in \{1,2\}$, let $C_i$ be a circuit in $\cC$ that spans $e_i$.
If $w \in X$ so that $\{w, e_1, e_2\}$ is a circuit of $M$, then $w$ is the unique element in $\cl_{M|X}(C_1 \cup C_2) - (\cl_{M|X}(C_1) \cup \cl_{M|X}(C_2))$.
\end{claim}
\begin{proof}
Certainly $w \in \cl_{M}(C_1 \cup C_2)$, since $\{w, e_1, e_2\}$ is a circuit.
If $w \in \cl_{M|X}(C_1)$, then $\cl_{M|X}(C_1)$ is a flat of $M|X$ that spans $e_2$, but does not contain either $2$-point line of $M|X$ contained in $C_2$, which contradicts that $e_2$ is of type-$(b)$.
Thus, $w \notin \cl_M(C_1)$, and, by symmetry, $w \notin \cl_M(C_2)$.
As $C_1 \cap C_2 = \{x,y\}$, it follows that $|\cl_{M|X}(C_1 \cup C_2)| = 10$ and $|\cl_{M|X}(C_1) \cup \cl_{M|X}(C_2)| = 9$.
Thus, the element $w$ is unique.
\end{proof}

The following deals with a slightly different case.

\begin{claim} \label{line and circuit}
Let $e \in E(M) - X$ be freely added to $L$, and let $f \in E(M) - (X \cup e)$ be spanned by a circuit $C \in \cC$.
If $w \in X$ so that $\{w, e, f\}$ is a circuit of $M$, then $w$ is the unique element in $\cl_{M|X}(C) - (L \cup C)$.
\end{claim}
\begin{proof}
Certainly $w \in \cl_{M|X}(C)$, since $C$ spans $L$ and $\{w,e,f\}$ is a circuit.
If $w \in L$, then $\cl_M(L)$ contains $\{x,y,z,e,f\}$, a contradiction.
If $w \in C$, then there is a $2$-point line of $M|X$ that contains $w$ and spans $e$ but does not contain $L$, which contradicts the fact that $e$ is freely added to $L$.
Since $|\cl_{M|X}(C)| = 6$ and $|L \cup C| = 5$, the element $w$ is unique.
\end{proof}

If $w \in X - L$ is in a circuit in $\cC$, then $w$ is incident to one of the vertices $1$ and $2$.
This implies that there is no pair $(L, C)$ or $(C_1, C_2)$ with $C, C_1, C_2 \in \cC$ so that $w$ is the unique element in $\cl_{M|X}(C) - (L \cup C)$ or $\cl_{M|X}(C_1 \cup C_2) - (\cl_{M|X}(C_1) \cup \cl_{M|X}(C_2))$.
So, by \ref{two circuits} and \ref{line and circuit}, there is no pair of elements in $E(M) - X$ that span $w$.
Since $w \notin L$ and each circuit in $\cC$ contains $x$ and $y$,  the element $w$ is in exactly one circuit in $\cC$.
Therefore, $w$ is on exactly one long line, $L'$, of $M$ that is not a long line of $M|X$.
Thus, $w$ is a tip of at most one spike in $M$, otherwise $M$ has a spike not using the line $L'$, so the graphic matroid $M|X$ has a spike, a contradiction.
Since $w \notin L$, it is not on a $4$-point line, so it is not special.

Since $z = 12$, we see that $z \in L$ and $z \in \cl_{M|X}(C)$ for all $C \in \cC$.
So, by \ref{two circuits} and \ref{line and circuit}, there is no pair of elements in $E(M) - X$ that span $z$.
This implies that each long line of $M$ through $z$ is spanned by a long line of $M|X$ through $z$, so $z$ is not a tip of a spike.
Since $z$ is on at most one $4$-point line, it is not special.

If $w \in X - L$ is not in a circuit in $\cC$, then there is at most one pair $(L, C)$ or $(C_1, C_2)$ with $C, C_1, C_2 \in \cC$ so that $w$ is the unique element in $\cl_{M|X}(C) - (L \cup C)$ or $\cl_{M|X}(C_1 \cup C_2) - (\cl_{M|X}(C_1) \cup \cl_{M|X}(C_2))$.
So, by \ref{two circuits} and \ref{line and circuit}, there is at most one pair of elements in $E(M) - X$ that span $w$.
Then, since $w$ is not in a circuit in $\cC$ and is not in $L$, it is on at most one long line of $M$ that is not a long line of $M|X$. Thus,   it is a tip of at most one spike.
Since $w \notin L$, it is not on a $4$-point line, so it is not special.
Therefore, $M$ has at most two special points in $X$.

Now, let $e \in E(M) - X$.
We will show that $e$ is not special.
We consider two cases.
First assume that $e$ is freely added to $L$.
As $E(M) - X$ is independent, by \ref{line and circuit}, each long line of $M$ through $e$ contains an element of $X$ that uses the vertex $0$.
Let $T$ be a transversal of the non-trivial parallel classes of $M/e$ so that each element in $T$ is in $X$ and uses the vertex $0$.
Then $T$ is independent in $M|X$.
Also, $T$ contains at most one of $01$ and $02$, and so $T$ does not span $L$ in $M|X$, which implies that $T$ does not span $e$ in $M$.
Since $T$ is independent in $M$ and does not span $e$, the element $e$ is not a tip of a spike restriction of $M$.
Since $e$ is on at most one $4$-point line of $M$, the element $e$ is not special.

Now assume that $e$ is not freely added to $L$.
Then $e$ is a type-$(b)$ element and is on no $4$-point lines of $M$.
Let $C$ be a circuit in $\cC$ so that both $2$-point lines of $M|X$ in $C$ span $e$.
We may assume that $C = \{01, 13, 23, 02\}$.
Then $\{01, 23\}$ and $\{02, 13\}$ both span $e$.
By \ref{line and circuit} and \ref{two circuits}, each long line of $M$ through $e$ contains an element of $X$ that uses the vertex $3$.
Let $T$ be a transversal of the non-trivial parallel classes of $M/e$ so that each element in $T$ is in $X$ and uses the vertex $3$. 
Then $T$ is independent in $M|X$.
Note that $T$ contains $13$ and $23$, and that $T - \{13\}$ does not span $\{01,23\}$ or $\{02,13\}$ in $M$.
Since $e$ is of type-$(b)$, this implies that $T - \{13\}$ does not span $e$.
Then $T - \{13\}$ is an independent set that does not span $e$, and is a transversal of all but one long line of $M$ through $e$.
Thus, $e$ is a tip of at most one spike.
Since $e$ is not on a $4$-point line of $M$, this implies that $e$ is not special.
Thus, $M$ has no special point in $E(M) - X$, and at most two special points in $X$, so the lemma holds.
\end{proof}

\end{section}

% spanning clique case

\begin{section}{The Spanning-Clique Case} \label{spanning-clique case}
We prove the following result for matroids with a spanning-clique restriction.
Recall that a point $x$ of a simple matroid $M$ is \emph{special} if it is on at least two $4$-point lines, is a tip of at least two spike restrictions of $M$, or is a tip of a spike restriction of $M$ and is on a $4$-point line.

\begin{proposition} \label{spanning clique}
Let $M$ be a simple matroid of rank at least six with a spanning-clique restriction.
If $M$ has no minor in $\{U_{2,5}, F_7, R_9, \U\}$, then 
\begin{enumerate}[$(i)$]
\item $|M| \le {r(M) + 2 \choose 2} -2$, and

\item $M$ has at most $21$ special points.
\end{enumerate} 
\end{proposition}
\begin{proof}
Let $M|X \cong M(K_{r(M)+1})$.
Note that $|M| \le {r(M) + 2 \choose 2} -2$ if and only if $|E(M) - X| \le r(M) - 1$.
By Proposition~\ref{projections}, each element of $E(M) - X$ is either freely added to a $3$-point line of $M|X$, or corresponds to the modular cut generated by a pair of $2$-point lines of $M|X$ whose union is a circuit.
Let $\cL$ be the set of $3$-point lines of $M|X$ that span an element in $E(M) - X$.
Let $\cC$ be the set of $4$-element circuits $C$ of $M|X$ for which both $2$-point lines of $M|X$ contained in $C$ span a common element of $E(M) - X$.

Clearly each line in $\cL$ spans exactly one element of $E(M) - X$, or else $M$ has a $U_{2,5}$-restriction.
If there is a pair $\{L, L'\}$ of $2$-point lines $L$ and $L'$ of $M|X$ whose union is a circuit so that $L$ and $L'$ each span elements $e$ and $f$ in $E(M) - X$, then $L \cup \{e,f\}$ and $L' \cup \{e,f\}$ are distinct lines of $M$ that intersect in two points, a contradiction.
Thus, $|E(M) - X| = |\cL| + |\cC|$.

We first deal with the case when two circuits in $\cC$ span each other.

\begin{claim} \label{double 2}
If there are distinct circuits $C_1$ and $C_2$ in $\cC$ whose union is contained in an $M(K_4)$-restriction of $M|X$, then $|E(M) - X| \le 3$, and $(i)$ and $(ii)$ hold.
\end{claim}
\begin{proof}
For each $i \in \{1,2\}$, let $e_i$ be the element of $M$ that is spanned by each $2$-point line of $M|X$ contained in $C_i$.
Since $C_1 \cup C_2$ is contained in an $M(K_4)$-restriction of $M|X$, there is a line $\{x, y\}$ of $M|X$ contained in both $C_1$ and $C_2$.

Suppose that $\cL \ne \varnothing$, and let $L \in \cL$.
Then $L \subseteq C_1 \cup C_2$, or else either $(L, C_1)$ or $(L, C_2)$ violates Lemma~\ref{2 and 3}.
Either $x$ or $y$ is in $L$, since every triangle of $K_4$ intersects every size-$2$ matching; assume without loss of generality that $x \in L$.
But then $x$ is on the $4$-point lines $\cl_M(L)$ and $\{x,y,e_1,e_2\}$, and also a $3$-point line other than $L$ in $C_1 \cup C_2$, a contradiction to Lemma~\ref{R}.
Thus, $\cL = \varnothing$.

If there is some circuit $C_3 \in \cC$ that is not contained in $C_1 \cup C_2$,
then there are at most two elements of $C_3$ in $C_1 \cup C_2$.
Since $|C_1 \cap C_3| = 2$ and $|C_2 \cap C_3| = 2$ by Lemma~\ref{2 and 2}, this implies that $C_3$ contains $\{x,y\}$.
But the only $4$-element circuits of $M|X$ that contain $\{x,y\}$ are $C_1$ and $C_2$.
Thus, each circuit in $\cC$ is contained in $C_1 \cup C_2$.
Since $C_1 \cup C_2$ is contained in an $M(K_4)$-restriction of $M|X$, it follows that $|\cC| \le 3$.
Then $|E(M) - X| \le 3$ so $(i)$ holds, and $(ii)$ holds by Lemma~\ref{3 case}.
\end{proof}

Next we treat the case when the lines in $\cL$ do not share a common element.

\begin{claim} \label{common edge}
If $\cL$ contains a set of three lines that do not intersect in a common point, then $|E(M) - X| = 3$, and $(i)$ and $(ii)$ hold.
\end{claim}
\begin{proof}
Since $M$ has no $(U_{2,4} \oplus U_{2,4})$-minor, each pair of lines in $\cL$ share a common element.
Since the lines in $\cL$ have no common point of intersection, there is a set $Z \subseteq X$ so that $M|Z \cong M(K_4)$ and $Z$ contains each line in $\cL$.
Suppose $|\cL| \ge 4$.
Then Lemma \ref{T_r case} shows that the lines in $\cL$ intersect in a common element, a contradiction.
Thus $|\cL| = 3$, and so there are three elements $x,y,z \in Z$ that are each on two lines in $\cL$.

Suppose $|\cC| \ge 1$.
By Lemma~\ref{2 and 3}, each circuit in $\cC$ is contained in $Z$, and so each circuit in $\cC$ contains one of $x,y,z$, since $|Z| = 6$.
But then one of $x,y,z$ is on two $4$-point lines and a $3$-point line of $M|\cl_M(Z)$, so $M$ has an $R_9$-restriction or a $U_{2,5}$-minor by Lemma~\ref{R}.
This is a contradiction, so $\cC = \varnothing$, and so $|E(M) - X| = 3$.
Then $|E(M) - X| \le 3$ so $(i)$ holds, and $(ii)$ holds by Lemma~\ref{3 case}.
\end{proof}

% three cases for $\cL$

We now consider four cases, depending on the size of $\cL$.
Suppose first that $|\cL| \ge 3$. 
By \ref{common edge}, the lines in $\cL$ share a common element $x$; thus, $|\cL| \le r(M) - 1$.
Since $|\cL| \ge 3$, the set $\cC$ is empty, or else there is a line $L \in \cL$ and a circuit $C \in \cC$ so that $(L,C)$ violates Lemma~\ref{2 and 3}.
Then $|E(M) - X| = |\cL| \le r(M) -1 $, and $(i)$ holds.
Also, $(ii)$ holds by Lemma~\ref{T_r case} if $|\cL| \ge 4$ and by Lemma~\ref{3 case} if $|\cL| = 3$.

Next assume that $|\cL| = 2$.
The union of the two lines in $\cL$ is contained in an $M(K_4)$-restriction $M|Z$ of $M|X$, and by Lemma~\ref{2 and 3}, each circuit in $\cC$ is contained in $Z$.
By Lemma~\ref{R}, the element $x$ that is on both lines in $\cL$ is in no circuit in $\cC$.
This implies that $|\cC| \le 1$, and so $|E(M) - X| \le 3 \le r(M) - 1$, and $(ii)$ holds by Lemma~\ref{3 case}.

Now suppose that $|\cL| = 1$.
Let $L \in \cL$, and let $e$ be the element of $E(M) - X$ that is spanned by $L$.
By \ref{double 2}, we may assume that each circuit in $\cC$ spans at most one element of $E(M) - X$ other than $e$.
%Let $\cK$ denote the collection of edge sets of $M(K_4)$-restrictions of $M|X$ that span an element of $E(M) - X$ other than $e$; then $|E(M) - X| = 1 + |\cK|$.
By Lemma~\ref{2 and 3}, each circuit in $\cC$ contains two elements of $L$.
If, for two different circuits $C_1$ and $C_2$ in $\cC$, we have $C_1 \cap L \ne C_2 \cap L$, then, as $r(C_1 \cup C_2) = 4$, we see that $|C_1 \cap C_2| = 1$, which contradicts Lemma~\ref{2 and 2}.
Thus, there are elements $x$ and $y$ in $L$ that are contained in each circuit in $\cC$.
This implies that $|\cC| \le r(M) - 2$, and so $|E(M) - X| \le r(M) - 1$.
Also, $(ii)$ holds by Lemma~\ref{T'_r case}.

Finally, suppose that $|\cL| = 0$.
By \ref{double 2}, we may assume that each $M(K_4)$-restriction of $M|X$ spans at most one element of $E(M) - X$.
By Lemma~\ref{2 and 2}, each pair of these $M(K_4)$-restrictions has a common triangle.
If there is no shared triangle among all of these restrictions, then there are exactly three $M(K_4)$-restrictions of $M|X$ that span an element of $E(M) - X$, and they are contained in an $M(K_5)$-restriction of $M|X$.
Then $|E(M) - X| = 3$, so $(i)$ and $(ii)$ hold, using Lemma~\ref{3 case}. 
Otherwise, there is a common triangle $L$ contained in each $M(K_4)$-restriction that spans an element of $E(M) - X$.
By the same reasoning at the end of the previous case, there are elements $x$ and $y$ in $L$ that are contained in each circuit in $\cC$, and so $(i)$ and $(ii)$ hold.
\end{proof}

\end{section}

% general case

\begin{section}{The Proofs of the Main Results} \label{proofs}
In this section, we prove Theorems \ref{main matroids} and \ref{main}.
We use two results of Geelen, Nelson, and Walsh from~\cite{complex}. 
To state these results, we need a notion of matroid connectivity.
A \emph{vertical $j$-separation} of a matroid $M$ is a partition $(X,Y)$ of $E(M)$ so that $r(X) + r(Y) - r(M) < j$ and $\min(r(X), r(Y)) \ge j$.
A matroid $M$ is \emph{vertically $k$-connected} if it has no vertical $j$-separation with $j<k$.
The first result~\cite[Theorem 6.1.3]{complex} finds structure in a minimal counterexample.
The second, which is a consequence of Theorem 7.6.1 in \cite{complex}, allows us to move from a clique minor to a spanning-clique restriction.

\begin{theorem}\label{new reduction}
There is a function $r_{\ref{new reduction}}\colon \bR^6\to \bZ$ so that, for all integers $\ell \ge 2$ and $t,s\ge 1$, and any real polynomial $p(x)=ax^2+bx+c$ with $a>0$, if $M\in\cU(\ell)$ satisfies $r(M)\ge r_{\ref{new reduction}}(a,b,c,\ell,t,s)$ and $\elem(M)>p(r(M))$, then $M$ has a minor $N$ with $\elem(N)>p(r(N))$ and $r(N)\ge t$ such that either
\begin{enumerate}[(1)]
\item $N$ has a spanning-clique restriction, or
\item $N$ is vertically $s$-connected and has an $s$-element independent set $S$ so that $\elem(N)-\elem(N/e)>p(r(N))-p(r(N)-1)$ for each $e\in S$.
\end{enumerate}
\end{theorem}

\begin{theorem} \label{spikes}
There are functions $s_{\ref{spikes}} \colon \bZ \to \bZ$ and $r_{\ref{spikes}} \colon \bZ^3 \to \bZ$ so that, for all integers $\ell \ge 2$, $k \ge 1$, and $m \ge 3$,
if $M$ is a vertically $s_{\ref{spikes}}(k)$-connected matroid with no $U_{2,\ell+2}$-minor and no Reid geometry minor, and with an $M(K_{r_{\ref{spikes}}(\ell, m, k)+1})$-minor and $2k$ special points, then $M$ has a simple minor $N$ of rank at least $m$ with an $M(K_{r(N)+1})$-restriction and $k$ special points.
\end{theorem}

In order to apply Theorem~\ref{spikes}, we need the following result of Geelen and Whittle~\cite{GW}.

\begin{theorem} \label{clique minor}
There is a function $\alpha_{\ref{clique minor}}\colon \bZ^2\to \bZ$ so that, for all integers $\ell,t\ge 2$, if $M\in \cU(\ell)$ satisfies $\elem(M)>\alpha_{\ref{clique minor}}(\ell,t) \cdot r(M)$, then $M$ has an $M(K_{t+1})$-minor.
\end{theorem}

%%% main proof

We now prove a restatement of Theorem~\ref{main matroids} and then Theorem~\ref{main}.

\begin{theorem} \label{restatement}
There is a constant $n_1$ so that if $M$ is a simple matroid of rank at least $n_1$ with no minor in $\{U_{2,5}, F_7, R_9, \U\}$, then $|M| \le {r(M) + 2 \choose 2} -2$.
\end{theorem}
\begin{proof}
Let $\cM$ be the class of matroids with no minor in $\{U_{2,5}, F_7, R_9, \U\}$.
Note that ${x + 2 \choose 2} - 2 = \frac{1}{2}x^2 + \frac{3}{2}x - 1$ for any positive number $x$.
Let $n_0$ be an integer so that ${x + 2 \choose 2} - 2 > \alpha_{\ref{clique minor}}(3, r_{\ref{spikes}}(3, 6, 22)) \cdot x$ for all $x \ge n_0$.
Define $$n_1 = r_{\ref{new reduction}}\Big(\frac{1}{2}, \frac{3}{2}, -1, 3, n_0, \max(s_{\ref{spikes}}(22), 44)\Big).$$
Assume that there is a simple matroid $M \in \cM$ with $r(M) \ge n_1$ and with  $|M| > {r(M) + 2 \choose 2} -2$.
By Theorem~\ref{new reduction} with $p(x) = {x + 2 \choose 2} - 2$ and $t = n_0$ and $s = \max\big(s_{\ref{spikes}}(22), 44\big)$, there is a minor $N$ of $M$ so that $r(N) \ge n_0$ and $|N| > {r(N) + 2 \choose 2} -2$, and either
\begin{enumerate}[(a)]
\item $N$ has a spanning-clique restriction, or

\item $N$ is vertically $s_{\ref{spikes}}(22)$-connected, and has a $44$-element independent set $S$ so that each $e \in S$ satisfies $\elem(N) - \elem(N/e) \ge r(N) + 2$.
\end{enumerate}

By Proposition~\ref{spanning clique}$(i)$, outcome (a) does not hold, so outcome (b) holds.

\begin{claim}
Each element in $S$ is special.
\end{claim}
\begin{proof}
For $e \in S$, let $\cL_N(e)$ denote the set of long lines of $N$ that contain $e$. 
Then 
\begin{align}
\elem(N) - \elem(N/e) = 1 + \sum_{L \in \cL_N(e)} \big(|L| - 2\big).
\end{align}
Let $N_1 = N|(\cup_{L \in \cL_N(e)}L)$, and note that $|\si(N_1/e)| = |\cL_N(e)|$.
If $e$ is on at least two $4$-point lines of $N$, then $e$ is special.
If $e$ is on no $4$-point line of $N$, then, since $\elem(N) - \elem(N/e) \ge r(N) + 2$, it follows from (1) that $|\cL_N(e) | \ge r(N) + 1$. 
Thus, $|\si(N_1/e)| \ge r(N) + 1$.
Since $\si(N_1/e)$ has rank at most $r(N) - 1$, this implies that $\si(N_1/e)$ has corank at least two, and so it contains at least two distinct circuits.
Then $e$ is a tip of two spike restrictions of $N$, and is thus special.
If $e$ is on exactly one $4$-point line of $N$, then (1) implies that $\si(N_1/e)$ has corank at least one and thus contains a circuit, so $e$ is a tip of a spike restriction of $N$.
\end{proof}

Since $r(N) \ge n_0$ and $|N| > {r(N) + 2 \choose 2} -2$, Theorem~\ref{clique minor} and the definition of $n_0$ imply that $N$ has an $M(K_{r_{\ref{spikes}}(3, 6, 22) + 1})$-minor.
Then, since $N$ has no $U_{2,5}$- or $R_9$-minor, Theorem~\ref{spikes} implies that $N$ has a simple minor of rank at least six with a spanning-clique restriction and at least $22$ special points, which contradicts Proposition~\ref{spanning clique}$(ii)$.
\end{proof}

\begin{proof}[Proof of Theorem~\ref{main}]
Let $n_1$ be the constant from Theorem~\ref{restatement}, and let $A$ be a $2$-modular matrix of rank $r \ge n_1$ with no zero-column and no parallel pair of columns.
Let $M$ be the vector matroid of $A$.
Then $M$ is a simple rank-$r$ matroid with no minor isomorphic to a matroid in $\{U_{2,5}, F_7, R_9, \U\}$.
By Theorem~\ref{restatement}, $|M| \le {r(M) + 2 \choose 2} - 2$, so $A$ has at most ${r(M) + 2 \choose 2} - 2$ columns.
\end{proof}
\end{section}

%\newpage

\end{document}